\documentclass[12pt,article]{amsart}
\usepackage[margin=1in]{geometry}
\usepackage{amsmath, amssymb,amsthm,url,mathrsfs,graphicx,amscd,amsfonts}
\usepackage[mathscr]{eucal}
\usepackage{dsfont}
\usepackage{mathtools}
\usepackage{hyperref}
\usepackage[normalem]{ulem} 
\usepackage[utf8]{inputenc}
\usepackage{graphicx,todonotes,hyperref}
\usepackage{epstopdf}
\usepackage{inputenc}
\usepackage{enumitem}
\usepackage{amsmath}
\usepackage{graphicx,todonotes,hyperref}
\usepackage{xcolor}
\usepackage{amssymb}

\newtheorem*{acknowledgements*}{Acknowledgements}

\allowdisplaybreaks
\newtheorem{theorem}{Theorem}[section]
\newtheorem{lemma}[theorem]{Lemma}
\newtheorem{corollary}[theorem]{Corollary}

\newtheorem*{definition}{Definition}

\newtheorem*{theorem*}{Theorem}
\theoremstyle{remark}
\newtheorem{remark}{Remark}

\numberwithin{equation}{section}
\newtheoremstyle{named}{}{}{\itshape}{}{\bfseries}{.}{.5em}{\thmname{#1 }\thmnote{#3}}
\theoremstyle{named}

\newcommand{\R}{\mathbb{R}}
\newcommand{\Z}{\mathbb{Z}}
\newcommand{\N}{\mathbb{N}}
\newcommand{\E}{\mathbb{E}}

\begin{document}
	
	\title[MPPC and additive energy]{Metric Poissonian pair correlation and additive energy}
	
	\author{Tanmoy Bera, E. Malavika}
	\address{Department of Mathematical Sciences, IISER Berhampur, Berhampur, Ganjam, Odisha, India 760003}
	\email{tanmoybera0165@gmail.com}
	\email{malavikae@iiserbpr.ac.in}

	\subjclass[2020]{11K06}
	\keywords{Pair correlation, additive energy, random zeta function, GCD sum}
	\maketitle
	\begin{abstract}
		In this article we prove that for a strictly increasing sequence $(a_n)$ of natural numbers, if the additive energy of $\{a_n:n\leq N\}$ is less than $N^3/(\log N)^C$ for some $C\geq14.71,$ then $(\{a_n\alpha\})$ has Poissonian pair correlation for almost all $\alpha\in\R.$ This provides a lower bound for the exponent $C$ in the additive energy bound established by Bloom and Walker~\cite{Bloom2019GCDSA}. 
	\end{abstract}
	
	\section{Introduction}
	Let $(x_n)\subseteq[0,1)$ be a sequence. We say that $(x_n)$ is uniformly distributed (UD) if
	\[\frac{1}{N}\#\{n\leq N: x_n\in[a,b)\}\to b-a,\text{ as }N\to\infty\]
	for any $0\leq a<b\leq 1$. The study of uniform distribution dates back to the seminal work of Weyl~\cite{Weyl}, and has been well studied over the last century. One of the well-known examples of UD sequence is $(\{n\alpha\})$ for any irrational $\alpha$.
	
	Let $s>0$ be a real number and $N$ be a natural number. The pair correlation function $R_2(s,(x_n),N)$ is defined by
	\[R_2(s,(x_n),N):=\frac{1}{N}\#\{(n,m)\in[1,N]^2\cap\N^2, n\neq m: \|x_n-x_m\|\leq s/N\},\] 
	where $\|x\|=\inf_{k\in\Z}|x+k|$ denotes the nearest integer distance. We say that $(x_n)$ has Poissonian pair correlation (PPC) if for any $s>0$
	\[R_2(s,(x_n),N)\to 2s,\text{ as } N\to \infty.\]
	This notion was first studied by Rudnick and Sarnak~\cite{RS}, where they showed that for any $d\geq 2,$ $(\{n^d\alpha\})$ has PPC for almost all $\alpha\in\R.$ The pair correlation statistics is a stronger property than uniform distribution, in the sense that if a sequence has PPC, then it is UD (see~\cite{CA2018uniform, Grepstad2017Larche, steinerberger2020poissonian}). 
	
	Let $(a_n)$ be an increasing sequence of natural numbers. For simplicity, throughout the article, let us denote the pair correlation function of $(\{a_n\alpha\})$ by $R_2(s,\alpha, N)$ instead of $R_2(s,(\{a_n\alpha\}), N)$.
	\begin{definition}
		We say that $(a_n)$ has metric Poissonian pair correlation (MPPC), if for almost all $\alpha\in\R,$ for any $s>0,$
		\[R_2(s,\alpha,N)\to 2s\text{ as }N\to\infty.\]
	\end{definition}
	
	Let $A\subseteq\N$ be a finite subset. The additive energy $E(A)$ of $A$ is defined by
	\[E(A):=\#\{(a,b,c,d)\in A^4:a+b=c+d\}.\]
	One can check that it satisfies
	$(\#A)^2\leq E(A)\leq (\#A)^3.$
	The representation function $r_A$ is defined in the following way:
	\[r_A(v):=\#\{(a,b)\in A^2: a-b=v\},\]
	where $v\in\Z.$ Then it is easily seen that the additive energy of $A$ satisfies the equality 
	$$E(A)=\sum_{v\in\Z}r_A(v)^2=\|r_A\|_2^2.$$
	Let $A_N:=\{a_n:n\leq N\}$ be the first $N$ terms of the sequence $(a_n).$ Recently, Aistleitner, Larcher and Lewko~\cite{aistleitner2017additive} showed that if 
	$$E(A_N)\ll N^{3-\delta}$$ 
	for some $\delta>0,$ then $(a_n)$ has MPPC. To the best of our knowledge, this is the first work to relate the metric Poissonian pair correlation property to the notion of additive energy. This result covers many sequences in one place, such as the result of Rudnick and Sarnak~\cite{RS}, because the additive energy of $(n^d)$ is $\ll N^{2+\epsilon}$ for any $\epsilon>0$ (see~\cite{aistleitner2017additive} for more examples). Later, Bloom and Walker~\cite{Bloom2019GCDSA} improved this result by relaxing the upper bound of the additive energy. They showed that if 
	$$E(A_N)\leq N^3/(\log N)^C$$ 
	for some large $C,$ then $(a_n)$ has MPPC. They conjectured that this holds for any $C$ strictly greater than $1.$ They have not provided any specific lower bound for $C.$ The purpose of this article is to give a lower bound of $C.$ Our main result is as follows.
	\begin{theorem}\label{main theorem}
		Let $(a_n)$ be a strictly increasing sequence of natural numbers. Let $C\geq 14.71$ be a real number. Suppose for all large $N$ the additive energy of $A_N$ satisfies 
		\[E(A_N)\leq N^3/(\log N)^C.\]
		Then $(a_n)$ has MPPC.
	\end{theorem}
	
	\subsection*{Remarks on the lower bound of C}
	\begin{remark}
		From the proof of the Theorem (see Section \ref{proof thm 1.1 section}), we get that $C$ must satisfy the inequality 
		\begin{align}\label{C condition}
			C-\beta-4\left\lceil\left(\sqrt{2C+1+1/10^{18}}\right)/2\right\rceil>1,
		\end{align}
		where $\beta=-2\sqrt{3}-6\log(1-1/\sqrt{3})\approx 1.7032$ and $\lceil x\rceil=\min\{n\in\Z\mid n\geq x\}$ is the ceiling function. It can be checked that only $C\in(\beta+13, 17.5-10^{-18}/2]\cup(\beta+17,\infty)$ satisfies~\eqref{C condition}. So, $\beta+13\approx14.71$ is the smallest $C$ that satisfies~\eqref{C condition}, i.e. for $C=14.71$ the conclusion of Theorem~\ref{main theorem} is true. Note that if $C> 14.71,$ then $E(A_N)\leq N^3/(\log N)^C$ implies $E(A_N)\leq N^3/(\log N)^{14.71}$. So, for any $C>14.71$, Theorem ~\ref{main theorem} holds whether it satisfies~\eqref{C condition} or not. 
	\end{remark}
	
	\begin{remark}
		In Theorem~\ref{gcd sum with representation}, $C\geq 7.5$ is needed, which gives a threshold to the lower bound of $C$ by this method. The moment lemma (Lemma~\ref{moment lemma}) with a lesser value of $\beta$ would decrease the exponent of $\log N$ in Theorem~\ref{gcd sum with representation}, resulting in a better lower bound of $C$ in Theorem~\ref{main theorem}. 
		Note that even if Lemma~\ref{moment lemma} is true for $\beta=1,$ the lower bound of $C$ in Theorem~\ref{main theorem} will be reduced by $0.7032.$ We have not been able to prove, but we believe that there is a very small room to improve the value of $\beta,$ which in turn suggests that the lower bound $14.71$ of $C$ can not be improved significantly using this method. We believe that only a different approach could significantly improve the lower bound of $C$, bringing it closer to the conjectured value of $1+\epsilon.$ 
	\end{remark}
	As an application of Theorem~\ref{main theorem}, we have 
	\begin{corollary}
		Let $K\geq 15.71$. Then $([n(\log n)^K])$ has MPPC.
	\end{corollary}
	This follows from the fact that for any $K\geq 1,$ the additive energy of $([n(\log n)^K])$ is $\ll N^3/(\log N)^{K-1}$ (see~\cite[Corollary 2]{Garaev2004}). 
	
	\section{Random Model}\label{rm sec}
	In this section, we discuss a random model introduced by Lewko and Radziwi{\l}{\l} in~\cite{lewko2017refinements}, which is used to study the bounds of GCD sums (see Section~\ref{sec gcd}). Let $\{X(p):p\text{ is prime}\}$ be a collection of independent random variables which are uniformly distributed in $\mathbb{S}^1.$ Let $\sigma>1/2$ be a real number. We define the random zeta function in the following way:
	\[\zeta_X(\sigma)=\prod_{p}\left(1-\frac{X(p)}{p^\sigma}\right)^{-1}.\]
	For any $\sigma>1/2,$ this series converges almost everywhere by the Kolmogorov three series theorem. For $n\in\N$ we define
	\begin{align}
		\label{X}
		X(n)=\prod_{p^r||n}X(p)^r.
	\end{align}
	Then in $L^p$ sense (with $p>0$) for $1/2<\sigma<1$
	\[\zeta_X(\sigma)
	=\sum_{n}\frac{X(n)}{n^\sigma},\]
	and 
	\[\E\left[X(n)\overline{X(m)}\right]=
	\begin{cases}
		1\quad& \text{ if } m=n\\
		0\quad &\text{ otherwise}
	\end{cases}\]
	where $\E(Y)=\int Y$ denotes the mean of a random variable $Y.$
	In the next lemma, we provide a more explicit (quantitative) upper bound of the $2\ell$-th moment of the random zeta function than that given in~\cite[Lemma 6]{lewko2017refinements}.
	\begin{lemma}[Moment lemma]\label{moment lemma}
		Let $1/2<\sigma<1/2+1/4$ and $\ell\geq 4.$ Then,
		$$\log\E[|\zeta_X(\sigma)|^{2\ell}]\leq (\ell^2+\beta \ell)\log((\sigma-1/2)^{-1}).$$
	\end{lemma}
	This plays a crucial role in establishing an upper bound for the GCD sum (see Section~\ref{sec gcd}). We state the following results, which will be used in the proof of Lemma~\ref{moment lemma}.
	
	\begin{lemma}\label{main lemma}
		Let $0<a\leq 1/\sqrt{3}$ be a real number and $\ell\in\N.$
		Then
		\[\int_{0}^{2\pi} 1/\left(1+a^2-2a\cos x\right)^\ell dx\leq\int_{0}^{2\pi}\exp(\beta \ell a^2+2\ell a\cos x)dx\] 
		for any $\beta\geq-2\sqrt{3}-6\log(1-1/\sqrt{3})\approx 1.7032$.
	\end{lemma}
	\begin{proof}
		We show that for all $x\in[0,2\pi],$ $\exp(\beta \ell a^2+2\ell a\cos x)\geq 1/\left(1+a^2-2a\cos x\right)^\ell,$ 
		which ensures the lemma. It is sufficient to prove the inequality for $\ell=1.$
		Let 
		\[f(x)=\exp(\beta a^2+2a\cos x)\left(1+a^2-2a\cos x\right)-1.\]
		The Lemma boils down to showing that $f(x)\geq 0$ for all $x\in[0,2\pi].$ Since $f(x)=f(2\pi-x),$ it reduces to show this for $x\in[0,\pi].$ Taking the first derivative, we get
		\[f'(x)=\exp(\beta a^2+2a\cos x)2a\sin x(2a\cos x-a^2).\]
		So, 
		$f'(x)\geq 0$ for all $x\in[0,\cos^{-1}(a/2)],$  and $f'(x)\leq 0$ for all $x\in[\cos^{-1}(a/2),\pi]$,
		i.e. $f$ increases in $[0,\cos^{-1}(a/2)], $ and decreases in $[\cos^{-1}(a/2),\pi].$ Therefore, it is sufficient to prove that $f(0),f(\pi)\geq 0.$
		
		Let $F(a)=\log ( f(0)+1)$. That is, $F(a)= 2a+\beta a^{2}+2\log(1-a).$
		Then $F'(a)=\frac{a}{1-a}(2\beta-2-2\beta a).$ Therefore, $F$ increases in $[0,1-1/\beta],$  and decreases in $[1-1/\beta,1).$ In the first case, i.e. for any $a\in [0,1-1/\beta], \ F(a)\geq F(0)=0.$ 
		Considering the second interval where $F$ decreases,  we choose the value of $\beta<2$ such that $F(1/\sqrt{3})\geq0.$ 
		\begin{align*}
			F(1/\sqrt{3})&=\beta/3+2/\sqrt{3}+2\log(1-1/\sqrt{3})\geq 0\\
			&\Leftrightarrow \beta\geq -2\sqrt{3}-6\log(1-1/\sqrt{3})\approx1.7032,
		\end{align*}
		which implies that $F(a)\geq0$ for all $0<a\leq 1/\sqrt{3}.$ Thus $f(0)\geq 0$ for these choices of $\beta.$ 
		Now we show $f(\pi)\geq0.$ In fact,
		\begin{align*}
			f(\pi)=&\exp(\beta a^2-2a)(1+a^2+2a)-1\geq 0\\
			\Leftrightarrow& (\beta -1)a^2\geq 0.
		\end{align*}
		Since $\beta>1,$ $f(\pi)\geq0.$
	\end{proof}

	\begin{lemma}\label{2lemma}
		Let $1/2<\alpha<1/2+1/4.$ Then
		\[\log((1-1/2^\alpha)^{-1})\leq 2.56/2^{2\alpha}.\] 
	\end{lemma}
	\begin{proof}
		Let $f(\alpha)=2.56/2^{2\alpha}+\log(1-1/2^\alpha).$ Then,
		\begin{align*}
			f'(\alpha)=&\frac{2^{2\alpha}-5.12(2^\alpha-1)}{2^{2\alpha}(2^\alpha-1)}\log 2\leq 0
		\end{align*}
		in $[1/2,1/2+1/4].$ Therefore, $f$ is decreasing in $[1/2,1/2+1/4],$ i.e. $f(\alpha)\geq f(3/4)>0.$ This completes the proof.
	\end{proof}
	
	Before starting the proof of Lemma~\ref{moment lemma}, we fix the value of $\beta.$ Throughout the subsequent sections $\beta=-2\sqrt{3}-6\log(1-1/\sqrt{3})\approx 1.7032.$
	\begin{proof}[Proof of Lemma~\ref{moment lemma}]
		Let $\ell$ be a natural number. We set
		$E_\ell(p):=\E\left[\left|\left(1-X(p)/p^\sigma\right)^{-2\ell}\right|\right].$ 
		Note that
		\[\log\E[|\zeta_X(\sigma)|^{2\ell}]=\sum_{p}\log E_\ell(p).\] 
		We aim to bound $E_\ell(p),$ which eventually leads to a bound of the $2\ell$-th moment of the random zeta function.
		Since $X(p)$ is uniformly distributed in $\mathbb{S}^1,$ for any $\ell\in\N$ and $p$ prime we have
		\begin{align*}
			E_\ell(p)=&\frac{1}{2\pi}\int_{0}^{2\pi}\left(1-\frac{e^{i\theta}}{p^\sigma}\right)^{-\ell}\left(1-\frac{e^{-i\theta}}{p^\sigma}\right)^{-\ell}d\theta\nonumber\\
			=&\frac{1}{2\pi}\int_{0}^{2\pi}\left(\frac{p^{2\sigma}}{p^{2\sigma}+1-2p^\sigma cos \theta}\right)^\ell d\theta\nonumber.
		\end{align*}
		By Lemma~\ref{main lemma}, for any prime $p\geq 3$ and $\ell\in\N$ we obtain
		\begin{align}\label{ine}
			E_\ell(p)&\leq\frac{1}{2\pi}\int_{0}^{2\pi}\exp(\beta \ell/p^{2\sigma}+2\ell/p^{\sigma}\cos x)dx\nonumber\\
			&= I_0(2\ell/p^\sigma)\exp(\beta \ell/p^{2\sigma}),
		\end{align}
		where $I_0$ is the $0$-th modified Bessel function (see~\cite[Page 713]{Arfken and Weber 2005}).
		Using $\log I_0(t)\leq t^2/4$ for all $t\geq 0$ in~\eqref{ine}, for any $\ell\in\N$ and prime $p\geq 3,$ we get
		\begin{align}\label{p3}
			\log E_\ell(p)\leq \frac{\ell^2+\beta \ell}{p^{2\sigma}}.
		\end{align}
		Now for $p=2$ and for any $\ell\in\N,$ we use the trivial upper bound of $E_\ell(2),$ which is $(1-1/2^\sigma)^{-2\ell}.$ For $1/2<\sigma<1/2+1/4,$ Lemma~\ref{2lemma} ensures that
		\begin{align}\label{p2}
			\log E_\ell(2)\leq 5.12\ell/2^{2\sigma}.
		\end{align}
		Invoking~\eqref{p3}, and~\eqref{p2} with $\ell\geq 4$ we get
		\begin{align}\label{log moment of random zeta }
			\log\E[|\zeta_X(\sigma)|^{2\ell}]&\leq(\ell^2+\beta \ell)\sum_{p}\frac{1}{p^{2\sigma}} \nonumber\\
			&\leq (\ell^2+\beta \ell)\sum_{p}\log((1-p^{-2\sigma})^{-1})\nonumber\\
			&=(\ell^2+\beta \ell)\log\zeta(2\sigma).
		\end{align}
		Thus proving $\log\zeta(2\sigma)\leq \log((\sigma-1/2)^{-1})$ would conclude the lemma.
		Recall that $$\zeta(2\sigma)=\frac{1}{2\sigma-1}+\sum_{n=0}^{\infty}\frac{(-1)^n\gamma_n}{n!}(2\sigma-1)^n,$$
		where $\gamma_0(\approx 0.577)$ is the Euler-Mascheroni constant and $\gamma_n$ is the $n$-th Stieltjes constant, which 
		satisfies the bound $|\gamma_n|\leq n!/2^{n+1}$ (see~\cite{Lavrik1976}). Using these bounds
		we have
		\[|\zeta(2\sigma)-1/(2\sigma-1)|<1\]
		in the given range $1/2<\sigma<1/2+1/4.$ Therefore,
		$\log\zeta(2\sigma)\leq\log((\sigma-1/2)^{-1}).$
	\end{proof}
	
	\section{GCD Sums}\label{sec gcd}
	Given a function $f:\mathbb{N}\rightarrow \mathbb{C}$ with finite support and $\sigma \in (0, 1]$, the GCD sum is defined as 
	\[S_f(\sigma):=\sum_{a, b\in\N} f(a)\overline{f(b)} \frac{\gcd(a, b)^{2\sigma}}{(ab)^\sigma}.\]
	This sum plays a major role in finding the large values of the Riemann zeta function (see~\cite{BS,BT}), and in the theory of Poissonian pair correlation (see~\cite{aistleitner2017additive, Bloom2019GCDSA}). In our next result, we provide an upper bound of $S_f(1/2)$ similar to Bloom and Walker~\cite[Theorem 2]{Bloom2019GCDSA}, but with an explicit constant.	\begin{theorem}\label{GCD Sum thm}
		Let $f:\N\to\R_{\geq 0}$ be a function of finite support with $\log \|f\|_1\geq 4$ and $K$ be a parameter such that $K\geq(\log\|f\|_1)^4.$ Let
		\[\sum_{\substack{a,b,c,d\in\N\\ab=cd}}f(a)f(b)f(c)f(d)\leq K\|f\|_2^4.\] 
		Then, 
		\begin{align*} 
			S_f(1/2)\ll&\frac{(\log \|f\|_1)^{\beta}}{\log\|f\|_1+\text{O}(1)}
			\exp\left(5/2\lceil((\log K)/(\log\log\|f\|_1))^{1/2}\rceil\log\log\|f\|_1\right)\|f\|_2^2.
		\end{align*}
	\end{theorem}
	\begin{proof}
		The proof is analogous to that of~\cite[Theorem 2]{Bloom2019GCDSA}, which uses the ideas of Lewko and Radziwi{\l}{\l}~\cite{lewko2017refinements}. Let $\sigma=1/2+1/\log\|f\|_1.$ Then, by applying H\"older's inequality we get
		\[S_f(1/2)\ll S_f(\sigma).\]  
		We define a random variable $D(X)$ as
		\[D(X)=\sum_{a\in\N}f(a)X(a),\]
		where $X(a)$ is as defined in~\eqref{X} (see Section~\ref{rm sec}).
		Note that 
		\begin{align}
			\label{D2}
			\E[|D(X)|^2]=\|f\|_2^2,
		\end{align}and 
		\begin{align}\label{4norm}
			\E[|D(X)|^4]=\sum_{\substack{a,b,c,d\in\N\\ab=cd}}f(a)f(b)f(c)f(d)\leq K\|f\|_2^4.
		\end{align} 
		It can be shown that
		\begin{align}\label{Moment and gcd sum}
			\E[|D(X)\zeta_X(\sigma)|^2]=\zeta(2\sigma)\sum_{a,b\in\N}f(a)f(b)\frac{\gcd(a,b)^{2\sigma}}{(ab)^\sigma}.
		\end{align}
		Let $V$ be a positive real number and $l$ be a natural number to be chosen later. Dividing the integral in~\eqref{Moment and gcd sum} into two parts depending on whether $|\zeta_X(\sigma)|\leq V$ or not, using~\eqref{D2} and applying the Cauchy-Schwarz inequality, we deduce that
		\begin{align}\label{ineq}
			\E[|D(X)\zeta_X(\sigma)|^2]&\leq \|f\|_2^2V^2+V^{-2\ell}(\E[|D(X)|^4])^{1/2}(\E[|\zeta_X(\sigma)|^{4+4\ell}])^{1/2}.
		\end{align}
		Invoking the moment lemma (Lemma~\ref{moment lemma}) and applying~\eqref{4norm} to the RHS of the above inequality~\eqref{ineq} we obtain
		\begin{align*}
			\E[|D(X)\zeta_X(\sigma)|^2]&\leq \|f\|_2^2\left(V^2+V^{-2\ell}K^{1/2}\exp((2(\ell+1)^2+\beta(l\ell+1))\log((\sigma-1/2)^{-1}))\right)\\
			&=\|f\|_2^2\left(V^2+V^{-2\ell}K^{1/2}\exp((2(\ell+1)^2+\beta(\ell+1))\log\log\|f\|_1)\right).
		\end{align*}
		Choosing
		\[V=\exp\left((\beta/2+5/4(\ell+1))\log\log\|f\|_1\right),\]
		\[\ell+1=\lceil((\log K)/(\log\log\|f\|_1))^{1/2}\rceil
		\]
		we get
		\begin{align}\label{2lth moment bound}
			\E[|D(X)\zeta_X(\sigma)|^2]\leq\|f\|_2^2(\log\|f\|_1)^\beta\exp\left(5/2\lceil((\log K)/(\log\log\|f\|_1))^{1/2}\rceil\log\log\|f\|_1\right).
		\end{align}
		It is to be noted that $\ell\geq 1,$ because of the assumption on $K.$
		Therefore,~\eqref{2lth moment bound} and~\eqref{Moment and gcd sum} together with $\zeta(2\sigma)=1/(2\sigma-1)+\text{O}(1)$ prove the theorem.
	\end{proof}
	
	Now we prove a particular case of the above theorem, which will be used to prove Theorem~\ref{main theorem}. The derived explicit exponent of $\log N$ in Theorem~\ref{gcd sum with representation} leads to a lower bound for $C.$ Such a conclusion was not attainable in~\cite[Theorem 5]{Bloom2019GCDSA} due to the absence of explicit constants in~\cite[Theorem 4]{Bloom2019GCDSA}.
	\begin{theorem}\cite[Theorem 3]{Bloom2019GCDSA}\label{thm3.2}
		Let $A\subseteq\N$ be a finite subset and $r=r_A$ be the representation function. Then, for some large constant $C_1$
		\[\sum_{\substack{a,b,c,d\in\N\\ab=cd}}r(a)r(b)r(c)r(d)\leq C_1|A|^6\log|A|.\]
	\end{theorem}
	\begin{theorem}\label{gcd sum with representation}
		Let $C\geq 15/2.$ Let $A\subseteq\N$ be a finite subset with $|A|=N.$ Assume that $E(A)\leq N^3/(\log N)^{C}.$ Then for large $N$ the GCD sum satisfies
		\begin{align*}
			\sum_{\substack{n_i,n_j\in A-A\\n_i,n_j>0}}\frac{\gcd(n_i,n_j)}{\sqrt{n_in_j}}r(n_i)r(n_j)\ll
			\quad N^3(\log N)^{\beta-1+4\left\lceil\left(\sqrt{2C+1+1/10^{18}}\right)/2\right\rceil-C}.
		\end{align*}
	\end{theorem}
	\begin{proof} 
		Until equation~\eqref {ineq} of Theorem~\ref{GCD Sum thm}, computations remain unchanged. By applying Theorem~\ref{thm3.2} and Lemma~\ref{moment lemma} with
		$f=r_{A}\chi_\N$ in~\eqref{ineq} we get
		\begin{align*}
			&\E[|D(X)\zeta_X(\sigma)|^2]\\
			&\leq E(A)V^2+V^{-2l}(C_1N^6\log N)^{1/2}\exp(((2l+2)^2+\beta(2l+2))\log\log N)^{1/2}\nonumber\\
			&\leq V^2N^3/(\log N)^C+V^{-2l}(C_1N^6\log N)^{1/2}\exp((2(l+1)^2)+\beta(l+1))\log\log N).
		\end{align*}
		We choose
		\[V=(\log N)^{\beta/2+2(\ell+1)}\]
		and
		\[\ell+1=\left\lceil\left(\sqrt{2C+1+1/10^{18}}\right)/2\right\rceil.\]
		This ensures that for large $N$
		\[V^2N^3/(\log N)^C\geq V^{-2l}(C_1N^6\log N)^{1/2}\exp((2(l+1)^2)+\beta(l+1))\log\log N).\]
		Therefore, by~\eqref{Moment and gcd sum} and the choice of $\sigma$ imply that the GCD sum is 
		\[\ll N^3(\log N)^{\beta-1+4\left\lceil\left(\sqrt{2C+1+1/10^{18}}\right)/2\right\rceil-C}.
		\]
	\end{proof}
	
	\section{Proof of Theorem~\ref{main theorem}}\label{proof thm 1.1 section}
	Let \[\text{Var}(R_2(s,\cdot,N)):=\int_0^{1}\left(R_2(s,\alpha,N)-2s(N-1)/N\right)^2d\alpha\]
	be the variance of $R_2(s,\cdot,N).$
	
	\begin{lemma}\cite[Lemma 3]{aistleitner2017additive}\label{boundlemm}
		Let $s>0.$ Then 
		\[\textnormal{Var}(R_2(s,\cdot,N))\ll\frac{s\log N}{N^3}\sum_{\substack{n_i,n_j\in A_N-A_N\\n_i,n_j>0}}\frac{\gcd(n_i,n_j)}{\sqrt{n_in_j}}r(n_i)r(n_j).\]
	\end{lemma}
	Applying Theorem~\ref{gcd sum with representation} in Lemma~\ref{boundlemm} we have
	\begin{align*}
		\text{Var}(R_2(s,\cdot,N)) \ll
		(\log N)^{\beta+4\left\lceil\left(\sqrt{2C+1+1/10^{18}}\right)/2\right\rceil-C}.
	\end{align*}
	Hence, the choice of $C$ for which $C-\beta-4\left\lceil\left(\sqrt{2C+1+1/10^{18}}\right)/2\right\rceil>1$
	gives
	\[\text{Var}(R_2(s,\cdot,N))\ll1/(\log N)^{1+\delta'}\]
	for some $\delta'>0.$
	The rest of the arguments are very standard (see~\cite{aistleitner2017additive,Bloom2019GCDSA}).
	Let $\eta>0$ be a real number to be chosen later. For $j\geq 1,$ let $N_j=[e^{\eta j}].$ Then 
	\begin{align*}
		\sum_{j\geq 1}\text{Var}(R_2(s,\cdot,N_j))<\infty.
	\end{align*} 
	By Chebyschev's inequality and the Borel–Cantelli Lemma, this implies that for almost all $\alpha$ there exists $j_0\in\N$ such that for all $j\geq j_0$
	\begin{align*}
		&\left|R_2\left(sN_j/N_{j+1},\alpha,N_j\right)-2s\right|<\epsilon,\hspace{1cm}|R_2(sN_{j+1}/N_{j},\alpha,N_{j+1})-2s|<\epsilon.
	\end{align*}
	Let $N\in\N$ be large and choose $j$ such that $N_j<N\leq N_{j+1}.$ One can easily check that 
	\begin{align*}
		N_jR_2\left(sN_j/N_{j+1},\alpha,N_j\right)\leq NR_2\left(s,\alpha,N\right)\leq N_{j+1}R_2\left(sN_{j+1}/N_{j},\alpha,N_{j+1}\right)
	\end{align*}
	Thus by $N_{j+1}/N_j=1+O(\eta)$ we have
	\[\left|R_2\left(s,\alpha,N\right)-2s\right|\leq\epsilon+O_s(\eta),\]
	for all large $N$ such that $j\geq j_0.$
	Finally, choosing $\eta$ sufficiently small completes the proof of the theorem.
	
	\section{Acknowledgement}
	The authors are thankful to Prof. Christoph Aistleiner for helpful discussions, which improved an earlier version of the article. Also, the authors would like to thank Dr. G. Kasi Viswanadham for many suggestions. The authors thank the anonymous referee for helpful comments and suggestions that improved the presentation of the paper.

\end{document}